\documentclass[twoside]{amsart}
\usepackage{latexsym}
\usepackage{amssymb,amsmath,amsopn}
\usepackage[dvips]{graphicx}   %To insert ps figures
\usepackage{color,epsfig}      %To insert ps figures
\usepackage{tikz} %for drawing figures
\usepackage{subcaption} %for putting a caption at figures?

\usepackage[percent]{overpic}

\usepackage{tikz}
\usepackage{tkz-euclide}

\usepackage[normalem]{ulem} %ez csak azért van, hogy át tudjak húzni szöveget \sout{}-al
\usepackage[utf8]{inputenc} % allow utf-8 input
\usepackage[T1]{fontenc}    % use 8-bit T1 fonts
\usepackage{amsfonts}       % blackboard math symbols

\DeclareMathOperator{\arccot}{arccot}
\DeclareMathOperator{\conv}{conv}
\DeclareMathOperator{\area}{area}
\DeclareMathOperator{\vol}{vol}

\DeclareMathOperator{\bd}{bd}

\DeclareMathOperator{\inter}{int}
\newcommand{\F}{\mathcal{F}}
\renewcommand{\Re}{\mathbb{R}}
\newcommand{\BB}{\mathbf{B}}
\newcommand{\B}{\inter (\BB^2 )}
\DeclareMathOperator{\hex}{hex}

\newtheorem{theorem}{Theorem}
\newtheorem{lemma}{Lemma}
\newtheorem{cor}{Corollary}
\newtheorem{prop}{Proposition}
\theoremstyle{definition}

\newtheorem{rem}{Remark}

\newtheorem{conj}{Conjecture}

\title[Minkowski arrangements]{On generalized Minkowski arrangements}
\author[M. Kadlicsk\'o]{M\'at\'e Kadlicsk\'o}
\author[Z. L\'{a}ngi]{Zsolt L\'{a}ngi$^*$}
\thanks{$^*$ Corresponsing author}
\address{Department of Geometry, Budapest University of Technology, Egry J\'{o}zsef utca 1., Budapest, Hungary, 1111}
\email{kadlicsko.mate@gmail.com}
\address{MTA-BME Morphodynamics Research Group and Department of Geometry, Budapest University of Technology, Egry J\'{o}zsef utca 1., Budapest, Hungary, 1111} 
\email{zlangi@math.bme.hu}

\thanks{The second named author is supported by the National Research, Development and Innovation Office, NKFI, K-119670, the J\'anos Bolyai Research Scholarship of the Hungarian Academy of Sciences, and the BME IE-VIZ TKP2020 and \'UNKP-20-5 New National Excellence Programs by the Ministry of Innovation and Technology.}

\begin{document}

\begin{abstract}
The concept of a Minkowski arrangement was introduced by Fejes T\'oth in 1965 as a family of centrally symmetric convex bodies with the property that no member of the family contains the center of any other member in its interior. This notion was generalized by Fejes T\'oth in 1967, who called a family of centrally symmetric convex bodies a generalized Minkowski arrangement of order $\mu$ for some $0 < \mu < 1$ if no member $K$ of the family overlaps the homothetic copy of any other member $K'$ with ratio $\mu$ and with the same center as $K'$. In this note we prove a sharp upper bound on the total area of the elements of a generalized Minkowski arrangement of order $\mu$ of finitely many circular disks in the Euclidean plane. This result is a common generalization of a similar result of Fejes T\'oth for Minkowski arrangements of circular disks, and a result of B\"or\"oczky and Szab\'o about the maximum density of a generalized Minkowski arrangement of circular disks in the plane. In addition, we give a sharp upper bound on the density of a generalized Minkowski arrangement of homothetic copies of a centrally symmetric convex body.
\end{abstract}

\maketitle

\section{Introduction}\label{sec:intro}

The notion of a \emph{Minkowski arrangement} of convex bodies was introduced by L. Fejes T\'oth in \cite{FTL}, who defined it as a family $\F$ of centrally symmetric convex bodies in the $d$-dimensional Euclidean space $\Re^d$, with the property that no member of $\F$ contains the center of any other member of $\F$ in its interior.
He used this concept to show, in particular, that the density of a Minkowski arrangement of homothets of any given plane convex body with positive homogeneity is at most four. Here an arrangement is meant to have positive homogeneity if the set of the homothety ratios is bounded from both directions by positive constants.
It is worth mentioning that the above result is a generalization of the planar case of the famous Minkowski Theorem from lattice geometry \cite{Lek}.
Furthermore, Fejes T\'oth proved in \cite{FTL} that the density of a Minkowski arrangement of circular disks in $\Re^2$ with positive homogeneity is maximal for a Minkowski arrangement of congruent circular disks whose centers are the points of a hexagonal lattice and each disk contains the centers of six other members on its boundary.

In \cite{agg}, extending the investigation to finite Minkowski arrangements, Fejes T\'oth gave a sharp upper bound on the total area of the members of a Minkowski arrangement of finitely many circular disks, and showed that this result immediately implies the density estimate in \cite{FTL} for infinite Minkowski circle-arrangements.  Following a different direction, in \cite{FTL2} for any $0 < \mu < 1$ Fejes T\'oth defined a \emph{generalized Minkowski arrangements of order $\mu$} as a family $\F$ of centrally symmetric convex bodies with the property that for any two distinct members $K, K'$ of $\F$, $K$ does not overlap the \emph{$\mu$-core} of $K'$, defined as the homothetic copy of $K'$ of ratio $\mu$ and concentric with $K'$. In this paper he made the conjecture that for any $0 < \mu \leq \sqrt{3}-1$, the density of a generalized Minkowski arrangement of circular disks with positive homogeneity is maximal for a generalized Minkowski arrangement of congruent disks whose centers are the points of a hexagonal lattice and each disk touches the $\mu$-core of six other members of the family. According to \cite{FTL2}, this conjecture was verified by B\"or\"oczky and Szab\'o in a seminar talk in 1965, though the first written proof seems to be published only in \cite{BSz1} in 2002. It was observed both in \cite{FTL2} and \cite{BSz1} that if $\sqrt{3}-1 < \mu < 1$, then, since the above hexagonal arrangement does not cover the plane, that arrangement has no maximal density.

In this paper we prove a sharp estimate on the total area of a generalized Minkowski arrangement of finitely many circular disks, with a characterization of the equality case. Our result includes the result in \cite{agg} as a special case, and immediately implies the one in \cite{BSz1}. The proof of our statement relies on tools from both \cite{BSz1, agg}, but uses also some new ideas.  In addition, we also generalize a result from Fejes T\'oth \cite{FTL} to find a sharp upper bound on the density of a generalized Minkowski arrangement of homothetic copies of a centrally symmetric convex body.

For completeness, we mention that similar statements for (generalized) Minkowski arrangements in other geometries and in higher dimensional spaces were examined, e.g. in \cite{BSz2, FTL3, molnar}. Minkowski arrangements consisting of congruent convex bodies were considered in \cite{BO}. Estimates for the maximum cardinality of mutually intersecting members in a (generalized) Minkowski arrangement can be found in \cite{foldvari, NS, NPS, polyanskii}. The problem investigated in this paper is similar in nature to those dealing with the volume of the convex hull of a family of convex bodies, which has a rich literature. This includes a result of Oler \cite{oler} (see also \cite{BL2}), which is also of lattice geometric origin \cite{zassenhaus}, and the notion of parametric density of Betke, Henk and Wills \cite{BHW}. In particular, our problem is closely related to the notion of density with respect to outer parallel domains defined in \cite{BL2}. Applications of (generalized) Minkowski arrangements in other branches of mathematics can be found in \cite{SW, wolansky}.

As a preliminary observation, we start with the following generalization of Remark 2 of \cite{FTL}, stating the same property for (not generalized) Minkowski arrangements of plane convex bodies.
In Proposition~\ref{genmink}, by $\vol_d(\cdot)$ we denote $d$-dimensional volume, and by $\BB^d$ we denote the closed Euclidean unit ball centered at the origin.

\begin{prop}\label{genmink}
Let $0 <\mu < 1$, let $K\subset \Re^d$ be an origin-symmetric convex body and let $\mathcal{F}=\{x_1+\lambda_1 K, x_2+\lambda_2K, \dots \}$ be a generalized Minkowski arrangement of order $\mu$, where $x_i \in \Re^d, \lambda_i>0$ for each $i=1,2, \dots$. Assume that $\F$ is of positive homogeneity, that is, there are constants $0 < C_1 < C_2$ satisfying $C_1 \leq \lambda_i \leq C_2$ for all values of $i$, and define the (upper) density $\delta(\mathcal{F})$ of $\mathcal{F}$ in the usual way as
\[
\delta(\F) = \limsup_{R \to \infty} \frac{ \sum_{x_i \in R \BB^d} \vol_d(x_i + \lambda_i K)}{\vol_d(R \BB^d)},
\]
if it exists.
Then
\begin{equation}\label{eq:genmink}
\delta(\F) \leq \frac{2^d}{(1+\mu)^d},
\end{equation}
where equality is attained, e.g. if $\{ x_1, x_2, \ldots \}$ is a lattice with $K$ as its fundamental region, and $\lambda_i = 2/(1+\mu)$ for all values of $i$.
\end{prop}

\begin{proof}
Note that the equality part of Proposition~\ref{genmink} clearly holds, and thus, we prove only the inequality in (\ref{eq:genmink}).
Let $|| \cdot ||_K: \Re^d\rightarrow  [ 0, \infty)$ denote the norm with $K$ as its unit ball. Then, by the definition of a generalized Minkowski arrangement, we have
\[
||x_i-x_j||_K \geq \max \{ \lambda_i + \mu \lambda_j , \lambda_j + \mu
\lambda_i \} \geq
\]
\[
\geq \frac{1}{2} \left( (\lambda_i + \mu \lambda_j)  +  (\lambda_j + \mu \lambda_i ) \right) = \frac{1+\mu}{2} (\lambda_i+\lambda_j),
\]
implying that the homothets $x_i + (\lambda_i/2) \cdot \left(1+\mu\right)  K$ are pairwise non-overlapping. In other words, the family
$\mathcal{F}'=\left\{ x_i + (\lambda_i/2) \cdot \left(1+\mu\right)  K:  i=1,2,\ldots \right\}$ is a packing. Thus, the density of $\mathcal{F}'$ is at most one,
from which (\ref{eq:genmink}) readily follows. Furthermore, if $K$ is the fundamental region of a lattice formed by the $x_i$'s and $\lambda_i = 2/(1+\mu)$ for all values of $i$, then $\mathcal{F}'$ is a tiling, implying the equality case.
\end{proof}

Following the terminology of Fejes T\'oth in \cite{agg} and to permit a simpler formulation of our main result, in the remaining part of the paper we consider generalized Minkowski arrangements of \emph{open} circular disks, where we note that generalized Minkowski arrangements can be defined for families of open circular disks in the same way as for families of closed circular disks.
%In our investigation, we denote the open unit circular disk centered at the origin by $\B^2$.

To state our main result, we need some preparation, where we denote the boundary of a set by $\bd (\cdot)$. Consider some generalized Minkowski arrangement $\F = \{ B_i = x_i + \rho_i \B : i=1,2,\ldots, n \}$  of open circular disks in $\Re^2$ of order $\mu$, where $0 < \mu < 1$.
Set $U(\F)= \bigcup_{i=1}^n B_i = \bigcup \F$. Then each circular arc $\Gamma$ in $\bd (U(\F))$ corresponds to a circular sector, which can be obtained as the union of the segments connecting a point of $\Gamma$ to the center of the disk in $\F$ whose boundary contains $\Gamma$. We call the union of these circular sectors the \emph{outer shell} of $\F$. Now consider a point $p \in \bd(U(\F))$ belonging to at least two members of $\F$, say $B_i$ and $B_j$, such that $x_i, x_j$ and $p$ are not collinear. Assume that the convex angular region bounded by the two closed half lines starting at $p$ and passing through $x_i$ and $x_j$, respectively, do not contain the center of another element of $\F$ in its interior which contains $p$ on its boundary. We call the union of the triangles $\conv \{ p,x_i,x_j \}$ satisfying these conditions the \emph{inner shell} of $\F$. We denote the inner and the outer shell of $\F$ by $I(\F)$ and $O(\F)$, respectively. Finally, we call the set $C(\F) = U(\F) \setminus (I(\F) \cup O(\F))$ the \emph{core} of $\F$ (cf. Figure~\ref{fig:core-shell}). Clearly, the outer shell of any generalized Minkowski arrangement of open circular disks is nonempty, but there are arrangements for which $I(\F) = \emptyset$ or $C(\F) = \emptyset$.

\begin{figure}[ht]
\begin{center}
\includegraphics[width=0.9\textwidth]{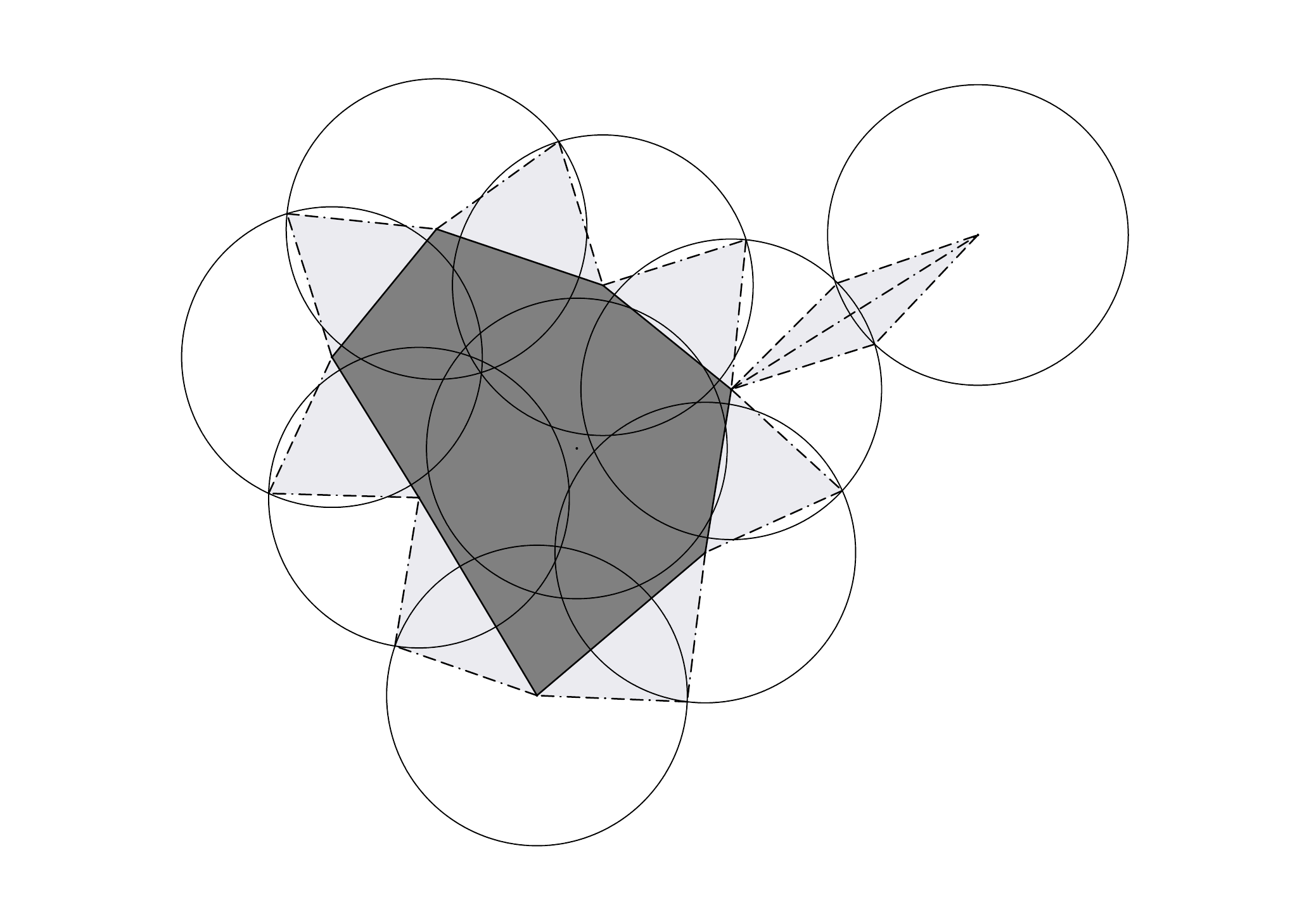} 
\caption{The outer and inner shell, and the core of an arrangement, shown in white, light grey and dark grey, respectively.}
\label{fig:core-shell}
\end{center}
\end{figure}

If the intersection of two members of $\F$ {is nonempty}, then we call this intersection a \emph{digon}. If a digon touches the $\mu$-cores of both disks defining it, we call the digon \emph{thick}. A digon which is not contained in a third member of $\F$ is called a \emph{free digon}.
Our main theorem is as follows, where $\area(X)$ denotes the area of the set $X$.

\begin{theorem}\label{thm:main}
Let $0 < \mu  \leq \sqrt{3}-1$, and let $\F = \{ B_i= x_i + \rho_i \B : i=1,2,\ldots, n \}$ be a generalized Minkowski arrangement of finitely many {open} circular disks of order $\mu$. Then
\[
T = \pi \sum_{i=1}^n \rho_i^2  \leq  \frac{2\pi }{\sqrt{3}(1+\mu)^2}\area(C(\mathcal{F}))+
\] 
\[
+\frac{4\cdot \arccos({\frac{1+\mu}{2})}}{(1+\mu)\cdot \sqrt{(3+\mu)(1-\mu)}}\area(I(\mathcal{F}))+\area(O(\mathcal{F})),
\]
where $T$ is the total area of the circles, with equality if and only if each free digon in $\F$ is thick. 
\end{theorem}

In the paper, for any points $x,y,z \in \Re^2$, we denote by $[x,y]$ the closed segment with endpoints $x,y$, by $[x,y,z]$ the triangle $\conv \{ x,y,z \}$, by $|x|$ the Euclidean norm of $x$, and if $x$ and $z$ are distinct from $y$, by $\angle xyz$ we denote the measure of the angle between the closed half lines starting at $y$ and passing through $x$ and $z$. Note that according to our definition, $\angle xyz$ is at most $\pi$ for any $x,z \neq y$.
% {By $\bd(\cdot)$ we denote the boundary of a set.} Furthermore,
For brevity we call {an open} circular disk a \emph{disk}, and a generalized Minkowski arrangement of disks of order $\mu$ a \emph{$\mu$-arrangement}. Throughout Sections~\ref{sec:prelim} and \ref{sec:proof} we assume that $0 < \mu \leq \sqrt{3}-1$.

In Section~\ref{sec:prelim}, we prove some preliminary lemmas. In Section~\ref{sec:proof}, we prove Theorem~\ref{thm:main}. Finally, in Section~\ref{sec:remarks}, we collect additional remarks and questions.

\section{Preliminaries}\label{sec:prelim}

For any $B_i, B_j \in \F$, if ${B_i \cap B_j} \neq \emptyset$, we call the two intersection points of $\bd(B_i)$ and $\bd(B_j)$ the \emph{vertices} of the digon $B_i \cap B_j$.

First, we recall the following lemma of Fejes T\'oth \cite[Lemma~2]{agg}. To prove it, we observe that for any $\mu > 0$, a generalized Minkowski arrangement of order $\mu$ is a Minkowski arrangement as well.

\begin{lemma}\label{2}
Let $B_i, B_j, B_k \in \F$ such that the digon $B_i \cap B_j$ is contained in $B_k$. Then the digon $B_i \cap B_k$ is free (with respect to $\F$).
\end{lemma}

From now on, we call the maximal subfamilies $\F'$ of $\F$ (with respect to containment) with the property that {$\bigcup_{B_i \in \F'} B_i$} is connected the \emph{connected components} of $\F$. Our next lemma has been proved by Fejes T\'oth in \cite{agg} for Minkowski arrangements of order $\mu=0$. His argument can be applied to prove Lemma~\ref{3} for an arbitrary value of $\mu$. Here we include this proof for completeness.
 
\begin{lemma}\label{3}
If $\F'$ is a connected component of $\F$ in which each free digon is thick, then the elements of $\F'$ are congruent.
\end{lemma}

\begin{proof}
We need to show that for any $B_i, B_j \in \F'$, $B_i$ and $B_j$ are congruent. Observe that by connectedness, we may assume that $B_i \cap B_j$ is a digon.
If $B_i\cap B_j$ is free, then it is thick, which implies that $B_i$ and $B_j$ are congruent.
If $B_i\cap B_j$ is not free, then there is a disk $B_k \in \F'$ containing it. By Lemma \ref{2}, the digons $B_i\cap B_k$ and $B_j\cap B_k$ are free. Thus $B_k$ is congruent to both $B_i$ and $B_j$. 
\end{proof}

In the remaining part of Section~\ref{sec:prelim}, we examine densities of some circular sectors in certain triangles. The computations in the proofs of these lemmas were carried out by a Maple 18.00 software. 

\begin{lemma}\label{shell_prelim}
Let $0 < \gamma < \pi$ and $A,B > 0$ be arbitrary. Let $T= [x,y,z]$ be a triangle such that $\angle xzy = \gamma$, and $|x-z|=A$ and $|y-z|=B$.
Let $\Delta=\Delta(\gamma,A,B)$, $\alpha=\alpha(\gamma,A,B)$ and $\beta=\beta(\gamma,A,B)$ denote the functions with variables $\gamma, A, B$ whose values are the area and the angles of $T$ at $x$ and $y$, respectively, and set $f_{A,B}(\gamma) = \left(\alpha A^2 + \beta B^2\right)/{\Delta}$. Then, for any $A,B > 0$, the function $f_{A,B}(\gamma)$ is strictly decreasing on the interval $\gamma \in (0,\pi)$.
\end{lemma}

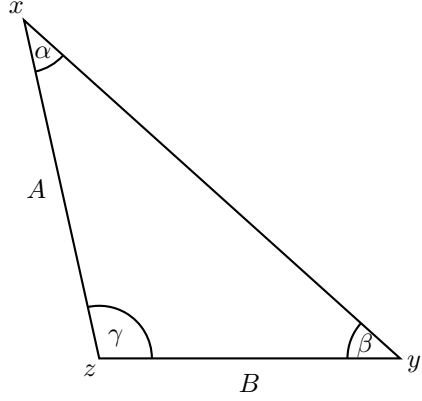
\begin{figure}
    \centering
    \begin{tikzpicture}[thick]
\coordinate (O) at (0,0);
\coordinate (A) at (4,0);
\coordinate (B) at (-1, 4.5);
\draw (O)--(A)--(B)--cycle;

\tkzLabelSegment[below=2pt](O,A){$B$}
\tkzLabelSegment[left=2pt](O,B){$A$}
\tkzLabelSegment[above right= 2pt ](A,B){}

\tkzMarkAngle[size=0.7, mark = none](A,O,B)
\tkzLabelAngle[pos = 0.35](A,O,B){$\gamma$}
\tkzLabelAngle[pos = -0.2](A,O,B){$z$}

\tkzMarkAngle[size=0.7, mark = none](B,A,O)
\tkzLabelAngle[pos = 0.5](B,A,O){$\beta$}
\tkzLabelAngle[pos = -0.2](B,A,O){$y$}

\tkzMarkAngle[size=0.7, mark = none](O,B,A)
\tkzLabelAngle[pos = 0.5](O,B,A){$\alpha$}
\tkzLabelAngle[pos = -0.2](O,B,A){$x$}

\end{tikzpicture}

    \caption{Notation in Lemma 3.}
    \label{fig:lm3}
\end{figure}
% ez a figure helyett (vagy mellett) szerintem lehet, hogy segítené az esetünket, ha leírnánk hogy az $f$ a sűrűsége köröknek a $T$ háromszögben --Máté
% Lehet. Mindenesetre a bírálóknak nem volt vele baja, később formálisan is elnevezzük, és igazából elég egyértelmű szerintem. Talán hagyjuk.

\begin{proof}
Without loss of generality, assume that $A \leq B$, and let $g = \alpha A^2 + \beta B^2$. Then, by an elementary computation, we have that
\[
g = A^2 \arccot \frac{A-B \cos \gamma}{B \sin \gamma} + B^2 \arccot \frac{B-A\cos \gamma}{A \sin \gamma},  \hbox{ and }  \Delta = \frac{1}{2} AB \sin \gamma.
\]
We regard $g$ and $\Delta$ as functions of $\gamma$.
We intend to show that $g' \Delta - g \Delta'$ is negative on the interval $(0,\pi)$ for all $A,B > 0$.
Let $h= g'\cdot \Delta/{\Delta'} - g$, and note that this expression is continuous on $\left(0, {\pi}/{2} \right)$ and $\left({\pi}/{2}, \pi \right)$ for all $A,B > 0$. By differentiating and simplifying, we obtain
\[
h'=\frac{-2\left( A^2 (1+\cos^2(\gamma)) + B^2 (1+\cos^2(\gamma)) - 4AB\cos(\gamma) \right) A^2 B^2 \sin^2(\gamma)}{\cos^2 \left( \gamma)(A^2+B^2-2AB \cos(\gamma) \right)^2},
\]
which is negative on its domain. This implies that $g' \Delta - g \Delta'$ is strictly decreasing on $\left(0, {\pi}/{2} \right)$ and strictly increasing on $\left( {\pi}/{2}, \pi \right)$. On the other hand, we have $\lim_{\gamma \to 0^+} \left( g' \Delta - g \Delta' \right) = -A^3 B \pi$, and $\lim_{\gamma \to \pi^-} \left( g' \Delta - g \Delta' \right) = 0$. This yields the assertion.
\end{proof}

\begin{lemma}\label{shell}
Consider two disks $B_i, B_j \in \F$ such that $|x_i-x_j| < \rho_i + \rho_j$, and let $v$ be a vertex of the digon $B_i \cap B_j$. Let $T = [x_i,x_j,v]$, $\Delta = \area(T)$, and let $\alpha_i = \angle vx_ix_j$ and $\alpha_j = \angle v x_j x_i$. Then
\begin{equation}\label{eq:shell}
\frac{1}{2}\alpha_i \rho_i^2+\frac{1}{2}\alpha_j \rho_j^2 \leq \frac{4 \arccos \frac{1+\mu}{2}}{(1+\mu)\sqrt{(1-\mu)(3+\mu)}} \Delta,
\end{equation}
with equality if and only if $\rho_i=\rho_j$ and $|x_i-x_j| = \rho_i(1+\mu)$.
\end{lemma}

\begin{figure}
    \centering
    \begin{tikzpicture}[thick]
\tkzDefPoint(0,0){A} 
\tkzDefPoint(3.7,0){B} 
\tkzInterCC[R](A, 2.0 cm)(B,2.9 cm) \tkzGetPoints{M1}{N1}

\tkzDrawCircle[R](A,2.cm) 
\tkzDrawCircle[R](B,2.9cm)

\coordinate (O) at (0,0);
\coordinate (A) at (3.7,0);

\draw (O)--(A)--(M1)--cycle;

\tkzLabelAngle[pos = -0.35](O,M1,A){$v$}

\tkzMarkAngle[size=0.7, mark = none](A,O,M1)
\tkzLabelAngle[pos = 0.4](A,O,M1){$\alpha_i$}
\tkzLabelAngle[pos = -0.35](A,O,M1){$x_i$}

\tkzMarkAngle[size=0.7, mark = none](M1,A,O)
\tkzLabelAngle[pos = 0.5](M1,A,O){$\alpha_j$}
\tkzLabelAngle[pos = -0.35](M1,A,O){$x_j$}

\tkzLabelSegment[above=2pt](O,M1){$\rho_i$}
\tkzLabelSegment[above=2pt](A,M1){$\rho_j$}

\tkzDrawPoints[color=black, fill=black](M1, O, A)

\end{tikzpicture}
    \caption{Notation in Lemma 4.}
    \label{fig:my_label}
    
\end{figure}
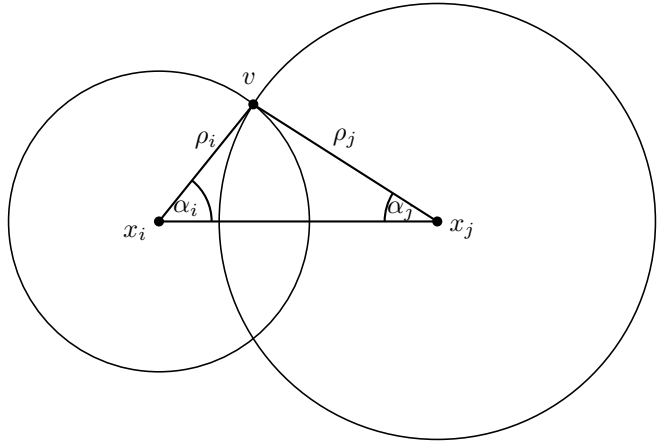

\begin{proof}
First, an elementary computation shows that if $\rho_i=\rho_j$ and $|x_i-x_j| = \rho_i(1+\mu)$, then there is equality in (\ref{eq:shell}).

Without loss of generality, let $\rho_i = 1$, and $0 < \rho_j = \rho \leq 1$. By Lemma~\ref{shell_prelim}, we may assume that $|x_i-x_j| = 1+\mu \rho$. Thus, the side lengths of $T$ are $1, \rho, 1+ \mu \rho$. Applying the Law of Cosines and Heron's formula to $T$ we obtain that
\[
\frac{\frac{1}{2}\alpha_i \rho_i^2+\frac{1}{2}\alpha_j \rho_j^2}{\Delta} = \frac{f(\rho,\mu)}{g(\rho,\mu)},
\]
where
\[
f(\rho,\mu) = \frac{1}{2} \arccos \frac{1+(1+\mu \rho)^2-r^2}{2(1+\mu \rho)}+ \frac{1}{2} \rho^2 \arccos \frac{\rho^2+(1+\mu \rho)^2-1}{2\rho(1+\mu \rho)},
\]
and
\[
g(\rho,\mu) = \rho \sqrt{2+\rho + \mu \rho)(2-\rho + \mu \rho)(1 - \mu^2)}.
\]
In the remaining part we show that
\[
\frac{f(\rho,\mu)}{g(\rho,\mu)} < \frac{4 \arccos \frac{1+\mu}{2}}{(1+\mu)\sqrt{(1-\mu)(3+\mu)}}
\]
if $0 < \rho < 1$ and $0 \leq \mu \leq \sqrt{3} -1$. To do it we distinguish two separate cases.

\emph{Case 1}, $0 < \rho \leq 1/5$. In this case we estimate $f(\rho,\mu)/g(\rho,\mu)$ as follows. Let the part of $[x_i,x_j]$ covered by both disks $B_i$ and $B_j$ be denoted by $S$. Then $S$ is a segment of length $(1-\mu) \rho$. On the other hand, if $A_i$ denotes the convex circular sector of $B_i$ bounded by the radii $[x_i,v]$ and $[x_i,x_j] \cap B_i$, and we define $A_j$ analogously, then the sets $A_i \cap A_j$ and $(A_i \cup A_j) \setminus T$ are covered by the rectangle with $S$ as a side which contains $v$ on the side parallel to $S$. The area of this rectangle is twice the area of the triangle $\conv (S \cup \{v \})$, implying that
\[
\frac{f(\rho,\mu)}{g(\rho,\mu)} \leq 1 + \frac{2(1-\mu)\rho}{1+\mu \rho}.
\]
We show that if $0 < \rho \leq 1/5$, then the right-hand side quantity in this inequality is strictly less than the right-hand side quantity in (\ref{eq:shell}).
By differentiating with respect to $\rho$, we see that as a function of $\rho$, $1 + \left(2(1-\mu)\rho \right)/(1+\mu \rho)$ is strictly increasing on its domain and attains its maximum at $\rho=1/5$. Thus, using the fact that this maximum is equal to $(7-\mu)/(5+\mu)$, we need to show that
\[
\frac{4 \arccos \frac{1+\mu}{2}}{(1+\mu)\sqrt{(1-\mu)(3+\mu)}} - \frac{7-\mu}{5+\mu} > 0.
\]
Clearly, the function
\[
\mu \mapsto \frac{ \arccos \frac{1+\mu}{2}}{ \frac{1+\mu}{2}}
\]
is strictly decreasing on the interval $[0,\sqrt{3}-1]$. By differentiation one can easily check that the function
\[
\mu \mapsto \frac{7-\mu}{5+\mu} \sqrt{(1-\mu)(3+\mu)}
\]
is also strictly increasing on the same interval. Thus, we obtain that the above expression is minimal if $\mu = \sqrt{3}-1$, implying that it is at least $0.11570\ldots$.

\emph{Case 2}, $1/5 < \rho \leq 1$. 

We show that in this case the partial derivative $\partial_{\rho} \left( f(\rho,\mu)/g(\rho,\mu) \right)$, or equivalently, the quantity $h(\rho,\mu) = f'_{\rho}(\rho,\mu) g(\rho,\mu)-g'_{\rho}(\rho,\mu) f(\rho,\mu)$ is strictly positive. By plotting the latter quantity on the rectangle $0 \leq \mu \leq \sqrt{3} - 1$, $1/5 \leq \rho \leq 1$, its minimum seems to be approximately $0.00146046085$. To use this fact, we upper bound the two partial derivatives of this function, and compute its values on a grid.
In particular, using the monotonicity properties of the functions $f,g$, we obtain that under our conditions $| f(\rho,\mu)| < 1.25$ and $|g(\rho,\mu)| \leq 0.5$. Furthermore, using the inequalities $0 \leq \mu \leq \sqrt{3} - 1$, $1/5 \leq \rho \leq 1$ and also the triangle inequality to estimate the derivatives of $f$ and $g$, we obtain that
\[
|f'_{\rho}(\rho,\mu)| < 1.95, \, |f'_{\mu}(\rho,\mu)| < 2.8, \, |f''_{\rho \rho}(\rho,\mu)| < 2.95, \, |f''_{\rho \mu}(\rho,\mu)| < 9.8,
\]
and
\[
|g'_{\rho}(\rho,\mu)| < 0.93, \, |g'_{\mu}(\rho,\mu)| < 1.08, \, |g''_{\rho \rho}(\rho,\mu)| < 2.64, \, |g''_{\rho \mu}(\rho,\mu)| < 15.1.
\]
These inequalities imply that $|h'_{\rho}(\rho,\mu)| < 4.78$ and $|h'_{\mu}(\rho,\mu)| < 28.49$, and hence,
for any $\Delta_{\rho}$ and $\Delta_{\mu}$, we have $h(\rho+\Delta_{\rho}, \mu + \Delta_{\mu}) > h(\rho,\mu) - 4.78 |\Delta_{\rho}| - 28.49 |\Delta_{\mu}|$.
Thus, we divided the rectangle $[0.2,1] \times [0,\sqrt{3}-1]$ into a $8691 \times 8691$ grid, and by numerically computing the value of $h(\rho,\mu)$ at the gridpoints, we showed that at any such point the value of $h$ (up to $12$ digits) is at least $0.00144$. According to our estimates above, this implies that $h(\rho,\mu) \geq 0.00002$ for all values of $\rho$ and $\mu$.
\end{proof}

Before our next lemma, recall that $\BB^2$ denotes the \emph{closed} unit disk centered at the origin.

\begin{lemma}\label{nocore}
For some $0 < \nu < 1$, let $x,y,z \in \Re^2$ be non-collinear points, and let $\{ B_u = u + \rho_u \BB^2 : u \in \{ x,y,z \} \}$ be a $\nu$-arrangement of disks; that is, assume that for any $\{ u,v \} \subset \{ x,y,z\}$, we have $|u-v| \geq \max \{ \rho_u,\rho_v \} + \nu \min \{ \rho_u,\rho_v\}$. Assume that for any $\{ u,v \} \subset \{ x,y,z \}$, $B_u \cap B_v \neq \emptyset$, and that the union of the three disks covers the triangle $[x,y,z]$. Then $\nu \leq \sqrt{3}-1$.
\end{lemma}
\begin{figure}
    \centering
    \begin{tikzpicture}[thick]
\tkzDefPoint(0,0){A} 
\tkzDefPoint(3.9,0){B}
\tkzDefPoint(1.3,-2.3){C}
\tkzInterCC[R](A, 2.0 cm)(B,2.9 cm) \tkzGetPoints{M1}{N1}

\tkzDrawCircle[R](A,2.0cm) 
\tkzDrawCircle[R, dotted](A, 1.0cm ) %\mu = 1/2
\tkzDrawCircle[R](B,2.9cm)
\tkzDrawCircle[R, dotted](B,1.45cm)
\tkzDrawCircle[R](C, 1cm)
\tkzDrawCircle[R, dotted](C, 0.5cm)

\coordinate (A) at (0,0);
\coordinate (B) at (3.9,0);
\coordinate (C) at (1.3, -2.3);

\draw (A)--(B)--(C)--cycle;

\tkzLabelAngle[pos = -0.35](C,A,B){$x$}
\tkzLabelAngle[pos = -0.35](A,B,C){$y$}
\tkzLabelAngle[pos = -0.35](A,C,B){$z$}

\tkzLabelSegment[above=2pt](B,C){$\rho_j$}

\tkzDrawPoints[color=black, fill=black](B,C, A)

%Sugarak
\tkzDefPoint(-1.41421356237, 1.41421356237){X}
\draw[line width= 0.5] (A)--(X);
\tkzLabelSegment[above=2pt](A,X){$\rho_x$}

\tkzDefPoint(3.9+(2.9)*1.41421356237/2, (2.9/2)*1.41421356237){Y}
\draw[line width= 0.5] (B)--(Y);
\tkzLabelSegment[above=2pt](B,Y){$\rho_y$}

\tkzDefPoint(1.3 - 1/2.2360679775, -2.3 -2/2.2360679775){Z}
\draw[line width= 0.5] (C)--(Z);
\tkzLabelSegment[above=2pt](C,Z){$\rho_z$}

\node at (1.5,-1) {$T$};
\end{tikzpicture}
    \caption{Notation in Lemma 5. The circles drawn with dotted lines represent the $\mu$-cores of the disks.}
    \label{fig:my_label2}
    
\end{figure}
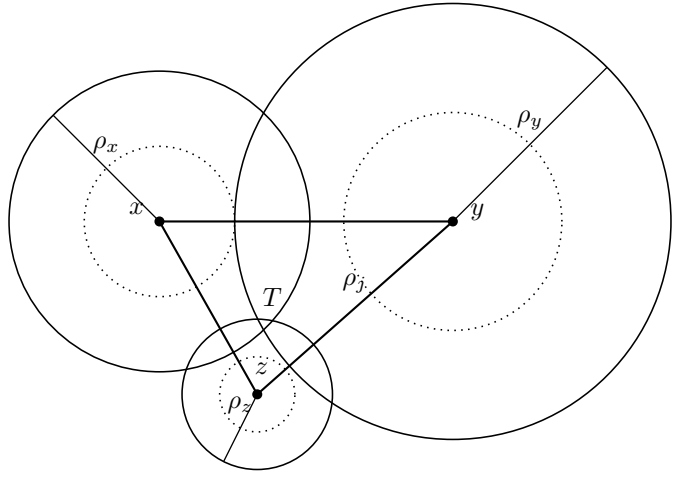
%It is worth remarking that the lemma does not hold without the condition that any two of the circles centered at $x,y,z$ intersect. Indeed, let $0 < \mu < 1$ arbitrary, and assume, for the moment, that $x,y,z$ are collinear points, with $y$ in the middle. Put two small circles centered at $x$ and $z$ and a large circle centered at $y$ such that the resulting triple is $\mu$-Minkowski. Then the segment $xz$ is contained in the interior of the union of the $3$ circles, and thus, moving $y$ slightly so that the $3$ points become non-collinear we obtain a $\mu$-Minkowski triple of circles so that their union covers the convex hull of their vertices.

\begin{proof}
Without loss of generality, assume that $0 < \rho_z \leq \rho_y \leq \rho_x$.
Since the disks are compact sets, by the Knaster-Kuratowski-Mazurkiewicz lemma \cite{KKM}, there is a point $q$ of $T$ belonging to all the disks, or in other words, there is some point $q \in T$ such that $|q-u| \leq \rho_u$ for any $u \in \{ x,y,z\}$. Recalling the notation $T=[x,y,z]$ from the introduction,
let $T' = [x',y',z' ]$ be a triangle with edge lengths $|y'-x'| = \rho_x + { \nu} \rho_y$, $|z'-x'| = \rho_x + { \nu} \rho_z$ and $|z'-y'|=\rho_y + { \nu} \rho_z$, and note that these lengths satisfy the triangle inequality.

We show that the disks $x'+\rho_x \BB^2$, $y'+\rho_y \BB^2$ and $z'+\rho_z \BB^2$ and $T'$ satisfy the conditions in the lemma. To do this, we show the following, more general statement, which, together with the trivial observation that any edge of $T'$ is covered by the two disks centered at its endpoints, clearly implies what we want: For any triangles $T= [x,y,z]$ and $ T'=[x',y',z']$ satisfying $|u'-v'| \leq |u-v|$ for any $u,v \in \{ x,y,z\}$, and for any point $q \in T$ there is a point $q' \in T'$ such that $|q'-u'| \leq |q-u|$ for any $u \in \{ x,y,z\}$. The main tool in the proof of this statement is the following straightforward consequence of the Law of Cosines, stating that if the side lengths of a triangle are $A,B,C$, and the angle of the triangle opposite of the side of length $C$ is $\gamma$, then for any fixed values of $A$ and $B$, $C$ is a strictly increasing function of $\gamma$ on the interval $(0,\pi)$.

To apply it, observe that if we fix $x,y$ and $q$, and rotate $[x,z]$ around $x$ towards $[x,q]$, we strictly decrease $|z-y|$ and $|z-q|$ and do not change $|y-x|, |z-x|, |x-q|$ and $|y-q|$. Thus, we may replace $z$ by a point $z^*$ satisfying $|z^*-y|=|z'-y'|$, or the property that $z^*, q,x$ are collinear. Repeating this transformation by $x$ or $y$ playing the role of $z$ we obtain either a triangle congruent to $T'$ in which $q$ satisfies the required conditions, or a triangle in which $q$ is a boundary point. In other words, without loss of generality we may assume that $q \in \bd (T)$. If $q \in \{ x,y,z\}$, then the statement is trivial, and so we assume that $q$ is a relative interior point of, say, $[x,y]$. In this case, if $|z-x| > |z'-x'|$ or $|z-y| > |z'-y'|$, then we may rotate $[y,z]$ or $[x,z]$ around $y$ or $x$, respectively. Finally, if $|y-x| > |y'-x'|$, then one of the angles $\angle yxz$ or $\angle xyz$, say $\angle xyz$, is acute, and then we may rotate $[z,y]$ around $z$ towards $[z,q]$. This implies the statement.

By our argument, it is sufficient to prove Lemma~\ref{nocore} under the assumption that $|y-x| = \rho_x + {\nu} \rho_y$, $|z-x| = \rho_x + { \nu} \rho_z$ and $|z-y|=\rho_y + { \nu} \rho_z$. Consider the case that $\rho_x > \rho_y$. Let $q$ be a point of $T$ belonging to each disk, implying that $|q-u| \leq \rho_u$ for all $u \in \{ x,y,z\}$. Clearly, from our conditions it follows that $|x-q| > \rho_x - \rho_y$. Let us define a $1$-parameter family of configurations, with the parameter $t \in [0,\rho_x-\rho_y]$, by setting $x(t) = x-tw$, where $w$ is the unit vector in the direction of $x-q$, $\rho_x(t)=\rho_x-t$, and keeping $q, y,z, \rho_y, \rho_z$ fixed.
Note that in this family $q \in B_{x(t)} = x(t) + \rho_x(t) \BB^2$, which implies that $|x(t)-u| \leq \rho_x(t)+\rho_u$ for $u \in \{ y,z \}$.
Thus, for any $\{ u,v \} \subset \{ x(t),y,z\}$, there is a point of $[u,v]$ belonging to both $B_u$ and $B_v$. This, together with the property that $q$ belongs 
to all three disks and using the convexity of the disks, yields that the triangle $[x(t),y,z]$ is covered by $B_{x(t)} \cup B_y \cup B_z$.
                
Let the angle between $u-x(t)$ and $w$ be denoted by $\varphi$. Then, using the linearity of directional derivatives, we have that for $f(t)= |x(t) -u|$, $f'(t) = - \cos \varphi \geq -1$ for $u \in \{y,z\}$, implying $|x(t)-u| \geq |x-u| - t = \rho_x(t) + \nu \rho_u$ for $u \in \{ y,z \}$, and also that the configuration is a $\nu$-arrangement for all values of $t$. 
Hence, all configurations in this family, and in particular, the configuration with $t = \rho_x-\rho_y$ satisfies the conditions in the lemma. Thus, repeating again the argument in the first part of the proof, we may assume that $\rho_x=\rho_y \geq \rho_z$, $|y-x| = (1+\mu) \rho_x$ and $|z-x| = |z-y| = \rho_x + { \nu} \rho_z$. Finally, if $\rho_x = \rho_y > \rho_z$, then we may assume that $q$ lies on the symmetry axis of $T$ and satisfies $|x-q| = |y-q| > \rho_x - \rho_z$. In this case we apply a similar argument by moving $x$ and $y$ towards $q$ at unit speed and decreasing $\rho_x=\rho_y$ simultaneously till they reach $\rho_z$, and, again repeating the argument in the first part of the proof, obtain that the family $\{ \bar{u} + \rho_z \BB^2 : \bar{u} \in \{ \bar{x},\bar{y},\bar{z} \} \}$, where $\bar{T}=[\bar{x},\bar{y},\bar{z}]$ is a regular triangle of side lengths $(1+ \nu) \rho_z$, covers $\bar{T}$. Thus, the inequality $ \nu \leq \sqrt{3}-1$ follows by an elementary computation.
\end{proof}

In our next lemma, for any disk $B_i \in \F$ we denote by $\bar{B}_i$ the closure $x_i + \rho_i \BB^2$ of $B_i$.

\begin{lemma}\label{core}
Let $B_i, B_j, B_k \in \F$ such that {$\bar{B}_u \cap \bar{B}_v \not\subseteq B_w$} for any $\{u,v,w \} = \{ i,j,k\}$.
Let $T=[x_i,x_j,x_k]$, $\Delta = \area(T)$, and $\alpha_u = \angle x_v x_u x_w$. %for any $\{u,v,w \} = \{ i,j,k\}$.
If $T \subset {\bar{B}_i \cup \bar{B}_j \cup \bar{B}_k}$, then
\begin{equation}\label{eq:core}
\frac{1}{2} \sum_{u \in \{i,j,k\}} \alpha_u \rho_u^2 \leq \frac{2\pi }{\sqrt{3}(1+\mu)^2} \Delta,
\end{equation}
with equality if and only if $\rho_i=\rho_j=\rho_k$, and $T$ is a regular triangle of side length $(1+\mu)\rho_i$.
\end{lemma}

\begin{proof}
In the proof we call $$\delta = \frac{\sum_{u \in \{i,j,k\}} \alpha_u \rho_u^2}{2\Delta}$$ the \emph{density} of the configuration.

Consider the $1$-parameter families of disks $B_u(\nu) = x_u + \left(1+\mu\right)/\left({1+\nu}\right) \rho_u \B$, where $u \in \{i,j,k\}$ and $\nu \in [\mu,1]$. Observe that the three disks $B_u(\nu)$, where $u \in \{ i,j,k \}$, form a $\nu$-arrangement for any $\nu \geq \mu$.
Indeed, in this case for any $\{ u,v \} \subset \{ i,j,k\}$, if $\rho_u \leq \rho_v$, we have 
\[
\frac{1+\mu}{1+\nu} \rho_v + \nu  \left(\frac{1+\mu}{1+\nu} \rho_u\right) =
\rho_v + \mu \rho_u - \frac{\nu-\mu}{1+\nu}(\rho_v-\rho_u) \leq \rho_v + \mu \rho_u \leq |x_u-x_v|.
\]
Furthermore, for any $\nu \geq \mu$, we have 
\[
(1+\mu)^2 \sum_{u \in \{i,j,k\}} \alpha_u \rho_u^2 =  (1+\nu)^2 \sum_{u \in \{i,j,k\}} \alpha_u \left( \frac{1+\mu}{1+\nu} \right)^2 \rho_u^2.
\]

Thus, it is sufficient to prove the assertion for the maximal value $\bar{\nu}$ of $\nu$ such that the conditions $T \subset \bar{B}_i(\nu) \cup \bar{B}_j(\nu) \cup \bar{B}_k(\nu)$ and $\bar{B}_u \cap \bar{B}_v \not\subseteq B_w$ are satisfied for any $\{u,v,w \} = \{ i,j,k\}$.  Since the relation $\bar{B}_u \cap \bar{B}_v\not\subseteq B_w$ implies, in particular, that $\bar{B}_u \cap \bar{B}_v \neq \emptyset$, in this case the conditions of Lemma~\ref{nocore}  are satisfied, yielding $\bar{\nu} \leq \sqrt{3}-1$. 
Hence, with a little abuse of notation, we may assume that $\bar{\nu}=\mu$. Then one of the following holds:
\begin{enumerate}
\item[(i)] The intersection of the disks $\bar{B}_u$ is a single point.
\item[(ii)] For some $\{u,v,w \} = \{ i,j,k\}$, $\bar{B}_u \cap \bar{B}_v \subset \bar{B}_w$ and $\bar{B}_u \cap \bar{B}_v \not\subset B_w$.
\end{enumerate}

Before investigating (i) and (ii), we remark that during this process, which we refer to as \emph{$\mu$-increasing process}, even though there might be non-maximal values of $\nu$ for which the modified configuration satisfies the conditions of the lemma and also (i) or (ii), we always choose the \emph{maximal} value. This value is determined by the centers of the original disks and the ratios of their radii.

First, consider (i). Then, clearly, the unique intersection point $q$ of the disks lies in $T$, and note that either $q$ lies in the boundary of all three disks, or two disks touch at $q$. We describe the proof only in the first case, as in the second one we may apply a straightforward modification of our argument. Thus, in this case we may decompose $T$ into three triangles $[x_i,x_j,q]$, $[x_i,x_k,q]$ and $[x_j,x_k,q]$ satisfying the conditions in Lemma~\ref{shell}, and obtain
\[
\frac{1}{2} \sum_{u \in \{i,j,k\}} \alpha_u \rho_u^2 \leq \frac{4 \arccos \frac{1+\mu}{2}}{(1+\mu)\sqrt{(1-\mu)(3+\mu)}} \Delta \leq \frac{2\pi }{\sqrt{3}(1+\mu)^2} \Delta,
\]
where the second inequality follows from the fact that the two expressions are equal if $\mu=\sqrt{3}-1$, and
\[
\left( 2 \arccos \frac{1+\mu}{2} - \frac{\pi \sqrt{(1-\mu)(3+\mu)}}{\sqrt{3} (1+\mu)} \right)' > 0
\]
if $\mu \in [0,\sqrt{3}-1]$. Here, by Lemma~\ref{shell}, equality holds only if $\rho_i=\rho_j=\rho_k$, and $T$ is a regular triangle of side length $(1+\mu)\rho_i$. On the other hand, under these conditions in (\ref{eq:core}) we have equality. This implies Lemma~\ref{core} for (i).

In the remaining part of the proof, we show that if (ii) is satisfied, the density of the configuration is strictly less than ${2\pi}/\left(\sqrt{3}(1+\mu)^2\right)$. Let $q$ be a common point of $\bd(\bar{B}_w)$ and, say, $\bar{B}_u$. If $q$ is a relative interior point of an arc in $\bd (\bar{B}_u \cap \bar{B}_v)$, then one of the disks is contained in another one, which contradicts the fact that the disks $B_u, B_v, B_w$ form a $\mu$-arrangement.
Thus, we have that either $\bar{B}_u \cap \bar{B}_v = \{ q \}$, or that $q$ is a vertex of the digon $B_u \cap B_v$. If $\bar{B}_u \cap \bar{B}_v = \{ q \}$, then the conditions of (i) are satisfied, and thus, we assume that $q$ is a vertex of the digon $B_u \cap B_v$.
By choosing a suitable coordinate system and rescaling and relabeling, if necessary, we may assume that $B_u = \B$, $x_v$ lies on the positive half of the $x$-axis, and $x_w$ is written in the form $x_w=(\zeta_w,\eta_w)$, where $\eta_w > 0$, and the radical line of $B_u$ and $B_v$ separates $x_v$ and $x_w$ (cf. Figure~\ref{fig:uvw}).
Set $\rho=\rho_w$. We show that $\eta_w > (1+\mu)\rho/{2}$.

\begin{figure}

\begin{center}
\begin{overpic}[width=0.7\textwidth]{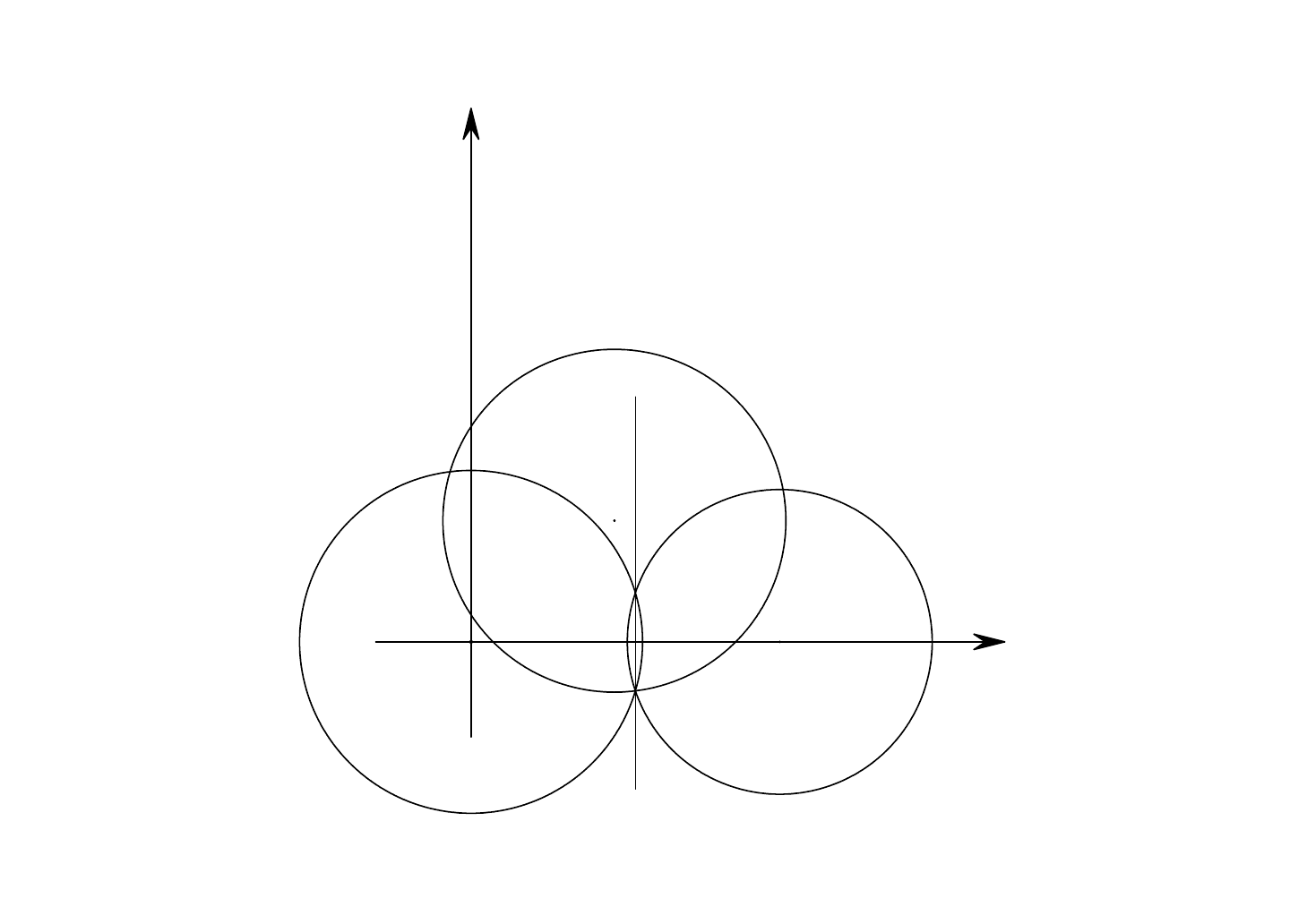}
 \put (30,23) {$\bar{B}_u$}
 \put (43, 33){$\bar{B}_w$}
 \put (60, 23){$\bar{B}_v$}
\end{overpic}
\caption{An illustration for the proof of Lemma~\ref{core}. }
\label{fig:uvw}
\end{center}
\end{figure}

\emph{Case 1}, if $\rho \geq 1$. Then we have $|x_w| \geq \rho + \mu$.

Let the radical line of $B_u$ and $B_v$ be the line $\{ x=t \}$ for some $0 < t  \leq 1$. Then, as this line separates $x_v$ and $x_w$, we have $\zeta_w \leq t$, and by (ii) we have $q=(t,-\sqrt{1-t^2})$. This implies that $|x_w-q| \leq |x_w-x_u|, |x_w-x_v|$, from which we have $0 \leq \zeta_w$. Let $S$ denote the half-infinite strip $S=\{ (\zeta,\eta) \in \Re^2 : 0 \leq \zeta \leq t, \eta \geq 0 \}$, and set $s=(t,-\sqrt{1-t^2}+\rho)$. Note that by our considerations, $x_w \in S$ and $|x_w - q| = \rho$, which yield $\eta_w \leq -\sqrt{1-t^2}+\rho$. From this it follows that $\rho + \mu \leq |x_w | \leq |s|$, or in other words, we have $t^2 + (\rho-\sqrt{1-t^2})^2 \geq (\rho + \mu)^2$. By solving this inequality for $t$ with parameters $\rho$ and $\mu$, we obtain that $t \geq t_0$, $1 \leq \rho \leq \left( 1-\mu^2\right)/\left(2\mu\right)$ and $0 \leq \mu \leq \sqrt{2}-1$, where
\[
t_0 = \sqrt{1- \left( \frac{1-2\mu \rho - \mu^2}{2 \rho} \right)^2}.
\]

Let $p=(\zeta_p,\eta_p)$ be the unique point in $S$ with $|p| = \rho+ \mu$ and $|p-q| = \rho$, and observe that $\eta_w \geq \eta_p$. Now we find the minimal value of $\eta_p$ if $t$ is permitted to change and $\rho$ is fixed. Set $p'=(\zeta_p,-\sqrt{1-\zeta_p^2})$. Since the bisector of $[p',q]$ separates $p'$ and $p$, it follows that $|p-p'| \geq |p-q|=\rho$ with equality only if $p'=q$ and $p=s$, or in other words, if $t=t_0$.
%, where we set $t_0 = \sqrt{1- \left( \frac{1-2\mu \rho - \mu^2}{2 \rho} \right)^2}$ ezt kitöröltem, mert a fenti jelölést használva felesleges --Máté
This yields that $\zeta_p$ is maximal if $t=t_0$. On the other hand, since $|p| = \rho+ \mu$ and $p$ lies in the first quadrant, $\eta_p$ is minimal if $\zeta_p$ is maximal. Thus, for a fixed value of $\rho$, $\eta_p$ is minimal if $t=t_0$ and $p = s = (t_0,-\sqrt{1-t_0^2}+\rho)$, implying that $\eta_w \geq -\sqrt{1-t_0^2}+\rho = \left( 2\rho^2+\mu^2 + 2\mu\rho-1\right) /(2\rho)$. Now, $\rho \geq 1$ and $\mu < 1$ yields that
\[
\frac{2\rho^2+\mu^2 + 2\mu\rho-1}{2\rho} - \frac{(1+\mu)\rho}{2} = \frac{\rho^2- \mu \rho^2+2\mu \rho - 1}{2\rho} \geq \frac{\mu}{2\rho} > 0,
\]
implying the statement.

\emph{Case 2}, if $0 < \rho \leq 1$. In this case the inequality $\eta_w > {(1+\mu)\rho}/{2}$ follows by a similar consideration.

In the remaining part of the proof, let 
$$\sigma(\mu) = \frac{2\pi}{\sqrt{3}(1+\mu)^2}.$$
Now we prove the lemma for (ii). Suppose for contradiction that for some configuration $\{ B_u, B_v, B_w \}$ satisfying (ii) the density is at least $\sigma(\mu)$; here we label the disks as in the previous part of the proof. Let $B'_w=x'_w + \rho_w \B$ denote the reflection of $B_w$ to the line through $[x_u,x_v]$.
By the inequality $\eta_w > {(1+\mu)\rho}/{2}$ proved in the two previous cases, we have that $\{ B_u, B_v, B_w, B_w' \}$ is a $\mu$-arrangement, where we observe that by the strict inequality, $B_w$ and $B_w'$ do not touch each others cores. Furthermore, each triangle $[x_u,x_w,x'_w]$ and $[x_v,x_w,x'_w]$ is covered by the three disks from this family centered at the vertices of the triangle, and the intersection of no two disks from one of these triples is contained in the third one. Thus, the conditions of Lemma~\ref{core} are satisfied for both $\{ B_u, B_w, B'_w\}$ and $\{ B_v, B_w, B'_w \}$. Observe that as by our assumption the density in $T$ is $\sigma(\mu)$, it follows that the density in at least one of the triangles $[x_u,x_w,x'_w]$ and $[x_v,x_w,x'_w]$, say in $T'=[x_u,x_w,x'_w]$, is at least $\sigma(\mu)$. In other words, under our condition there is an axially symmetric arrangement with density at least $\sigma(\mu)$. Now we apply the $\mu$-increasing process as in the first part of the proof and obtain a $\mu'$-arrangement $\{ \hat{B}_u = x_u + {(1+\mu)}/{(1+\mu')} \rho_u \B, \hat{B}_w = x_w + {(1+\mu)}/{(1+\mu')} \rho_w \B, \hat{B}_w' = x_w' + {(1+\mu)}/{(1+\mu')} \rho_w \B \}$ with density ${\sigma(\mu')}$ and $\mu' \geq \mu$ that satisfies either (i) or (ii). If it satisfies (i), we have that the density of this configuration is at most $\sigma(\mu')$ with equality if only if $T'$ is a regular triangle of side length $(1+\mu')\rho$, where $\rho$ is the common radius of the three disks. On the other hand, this implies that in case of equality, the disks centered at $x_w$ and $x_w'$ touch each others' cores which, by the properties of the $\mu$-increasing process, contradicts the fact that $B_w$ and $B_w'$ do not touch each others' $\mu$-cores. Thus, we have that the configuration satisfies (ii).

From Lemma~\ref{2} it follows that $\hat{B}_w \cap \hat{B}_w' \subset \hat{B}_u$. Thus, applying the previous consideration with $\hat{B}_u$ playing the role of $B_w$, we obtain that the distance of $x_u$ from the line through $[x_w,x_w']$ is greater than ${(1+\mu')}/{2} \rho_u$. Thus, defining $\hat{B}_u'=x_u'+ {(1+\mu)}/{(1+\mu')} \rho_u \B$ as the reflection of $B_u'$ about the line through $[x_w,x_w']$, we have that $\{ \hat{B}_u, \hat{B}_w, \hat{B}_w', \hat{B}_u' \}$ is a $\mu'$-arrangement such that $\{ \hat{B}_u, \hat{B}_u', \hat{B}_w \}$ and $\{ \hat{B}_u, \hat{B}_u', \hat{B}_w' \}$ satisfy the conditions of Lemma~\ref{core}. Without loss of generality, we may assume that the density of $\{ \hat{B}_u, \hat{B}_u', \hat{B}_w \}$ is at least $\sigma(\mu')$. Again applying the $\mu$-increasing procedure described in the beginning of the proof, we obtain a $\mu''$-arrangement of three disks, with $\mu'' \geq \mu'$, concentric with the original ones that satisfy the conditions of the lemma and also (i) or (ii). Like in the previous paragraph, (i) leads to a contradiction, and we have that it satisfies (ii). Now, again repeating the argument we obtain a $\mu'''$ -arrangement
\[
\left\{ y + \frac{1+\mu}{1+\mu'''} \rho_u \B, x_w + \frac{1+\mu}{1+\mu''''} \rho_w \B, x_w' + \frac{1+\mu}{1+\mu'''} \rho_w \B \right\},
\]
with density at least $\sigma(\mu''')$ and $\mu''' \geq \mu''$, that satisfies the conditions of the lemma, where either $y=x_u$ or $y = x'_u$. On the other hand, since in the $\mu$-increasing process we choose the \emph{maximal} value of the parameter satisfying the required conditions, this yields that $\mu'=\mu''=\mu'''$. But in this case the property that $\{ \hat{B}_u, \hat{B}_u', \hat{B}_w \}$ satisfies (ii) yields that $\{ \hat{B}_u, \hat{B}_u', \hat{B}_w \}$ does not; a contradiction.
\end{proof}

\section{Proof of Theorem ~\ref{thm:main}}\label{sec:proof}

The idea of the proof follows that in \cite{agg} with suitable modifications. In the proof we decompose $U(\F)=\bigcup_{i=1}^n B_i$, by associating a polygon to each vertex of certain free digons formed by two disks. Before doing it, we first prove some properties of $\mu$-arrangements.

Let $q$ be a vertex of a free digon, say, $ D= B_1 \cap B_2$.
We show that the convex angular region $R$ bounded by the closed half lines starting at $q$ and passing through $x_1$ and $x_2$, respectively, does not contain the center of any element of $\F$ different from $B_1$ and $B_2$ containing $q$ on its boundary. 
Indeed, suppose for contradiction that there is a disk $B_3=x_3+\rho_3 \B \in \F$ with $q \in {\bd (B_3) }$ and $x_3 \in R$. Since $[q,x_1,x_2] {\setminus \{ q \} }  \subset B_1 \cup B_2$, from this and the fact that $\F$ is a Minkowski-arrangement, it follows that the line through $[x_1,x_2]$ strictly separates $x_3$ from $q$. As this line is the bisector of the segment $[q,q']$, where $q'$ is the vertex of $D$ different from $q$, from this it also follows that $|x_3-q|>|x_3-q'|$. Thus, $q' \in B_3$.

 Observe that in a Minkowski arrangement any disk intersects the boundary of another one in an arc shorter than a semicircle. This implies, in particular, that $B_3 \cap \bd(B_1)$ and $B_3 \cap \bd(B_2)$ are arcs shorter than a semicircle. On the other hand, from this the fact that $q,q' \in B_3$ yields that $\bd(D) \subset B_3$, implying, by the properties of convexity, that $D \subset B_3$, which contradicts our assumption that $D$ is a free digon.
 
 Note that, in particular, we have shown that if a member of $\F$ contains both vertices of a digon, then it contains the digon.

\begin{figure}[ht]
\begin{center}
\begin{overpic}[width=0.5\textwidth]{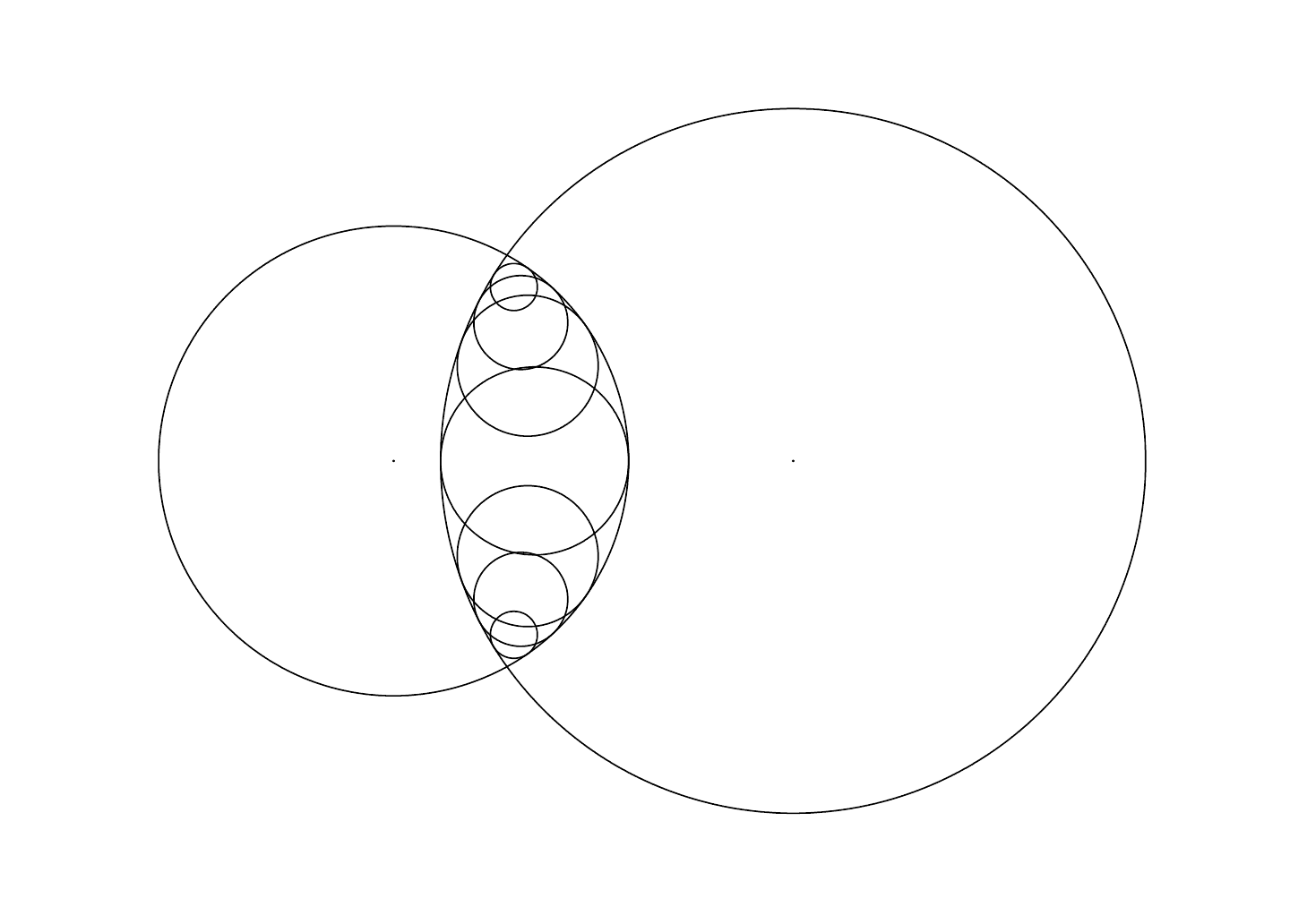}
\put(22, 28){$B_1$}
\put(65, 28){$B_2$}
\put(37, 34){$B(t)$}
\end{overpic}
\caption{The $1$-parameter family of disks inscribed in $B_1 \cap B_2$.}
\label{fig:inscribed}
\end{center}
\end{figure}

Observe that the disks inscribed in $D$ can be written as a $1$-parameter family {of disks} $B(t)$ continuous with respect to Hausdorff distance, where $t \in (0,1)$ {and $B(t)$ tends to $\{ q \}$ as $t \to 0^+$} (cf. Figure~\ref{fig:inscribed}); here the term `inscribed' means that the disk is contained in $B_i \cap B_j$ and touches both disks from inside.
We show that if some member $B_k$ of $\F$, different from $B_1$ and $B_2$, contains $B(t)$ for some value of $t$, then $B_k$ contains exactly one vertex of $D$.
% Egy kicsit pongyola volt a 'some $B(t)$s' szöveg, de láttam már ilyet cikkben. 
Indeed, assume that some $B_k$ contains some $B(t)$ but it does not contain any vertex of $D$. Then for $i \in \{1,2\}$, $B_k \cap \bd(B_i)$ is {a circular} arc $\Gamma_i$ in $\bd (D)$. Let $L_i$ be the half line starting at the midpoint of $\Gamma_i$, and pointing in the direction of the outer normal vector of $B_i$ at this point. Note that as $D$ is a plane convex body, $L_1 \cap L_2 = \emptyset$. On the other hand, since $B_1,B_2, B_k$ are a Minkowski arrangement, from this it follows that $x_k \in L_1 \cap L_2$; a contradiction. The property that no $B_k$ contains both vertices of $D$ follows from the fact that $D$ is a free digon.
Thus, if $q \in B_k$ for an element $B_k \in \F$, then there is some value $t_0 \in {(0,1)}$ such that $B(t) \subseteq B_k$ if and only if $t \in (0,t_0]$.

In the proof, we call the disks $B_i, B_j$ \emph{adjacent}, if $B_i \cap B_j$ is a digon, and there is a member of the family $B(t)$ defined in the previous paragraph
that is not contained in any element of $\F$ different from $B_i$ and $B_j$. Here, we remark that any two adjacent disks define a free digon, and if a vertex of a free digon is a boundary point of $U(\F)$, then the digon is defined by a pair of adjacent disks.

Consider a pair of adjacent disks, say $B_1$ and $B_2$, and let $q$ be a vertex of $D = B_1 \cap B_2$. 
If $q$ is a boundary point of the union $U(\F)$, then we call the triangle $[x_1,x_2,q]$ a \emph{shell triangle}, and observe that by the consideration in the previous paragraph, the union of shell triangles coincides with the inner shell of $\F$.

If $q$ is not a boundary point of $U(\F)$, then there is a maximal value $t_0\in {(0,1)}$ such that $B(t_0)=x+\rho\BB^2$ is contained in an element $B_i$ of $\F$ satisfying $q \in B_i$. Then, clearly, $B(t_0)$ touches any such $B_i$ from inside, and since $B_1$ and $B_2$ are adjacent, there is no element of $\F$ containing $B(t_0)$ and the vertex of $D$ different from $q$. Without loss of generality, assume that the elements of $\F$ touched by $B(t_0)$ from inside are $B_1, B_2, \ldots, B_k$. Since $B_1$ and $B_2$ are adjacent and there is no element of $\F$ containing both $B(t_0)$ and the vertex of $D$ different from $q$, we have that the tangent points of $B_1$ and $B_2$ on $\bd(B(t_0))$ are consecutive points among the tangent points of all the disks $B_i$, where $1 \leq i \leq k$. Thus, we may assume that the tangent points of $B_1, B_2, \ldots, B_k$ on $B(t_0)$ are in this counterclockwise order on $\bd(B(t_0))$.
Let $x$ denote the center of $B(t_0)$. Since $\F$ is a Minkowski arrangement, for any $1 \leq i < j \leq k$, the triangle $[x,x_i,x_j]$ contains the center of no element of $\mathcal{F}$ apart from $B_i$ and $B_j$, which yields that the points $x_1, x_2, \ldots, x_k$ are in convex position, and their convex hull $P_q$ contains $x$ in its interior but it does not contain the center of any element of $\F$ different from $x_1,x_2,\ldots, x_k$ (cf. also \cite{agg}). We call $P_q$ a \emph{core polygon}.

We remark that since $\F$ is a $\mu$-arrangement, the longest side of the triangle $[x,x_i,x_{i+1}]$, for $i=1,2\ldots,k$, is $[x_i,x_{i+1}]$. This implies that $\angle x_i x x_{i+1} > {\pi}/{3}$, and also that $k < 6$. Furthermore, it is easy to see that for any $i=1,2,\ldots,k$, the disks $B_i$ and $B_{i+1}$ are adjacent.
Thus, any edge of a core polygon is an edge of another core polygon or a shell triangle.
This property, combined with the observation that no core polygon or shell triangle contains any center of an element of $\F$ other than their vertices, implies that
core polygons cover the core of $\F$ without interstices and overlap (see also \cite{agg}).

Let us decompose all core polygons of $\F$ into triangles, which we call \emph{core triangles}, by drawing all diagonals in the polygon starting at a fixed vertex, and note that the conditions in Lemma~\ref{core} are satisfied for all core triangles. Now, the inequality part of Theorem~\ref{thm:main} follows from Lemmas~\ref{shell} and \ref{core}, with equality if and only if each core triangle is a regular triangle $[x_i,x_j,x_k]$ of side length $(1+\mu) \rho$, where $\rho=\rho_i=\rho_j=\rho_k$, and each shell triangle $[x_i,x_j,q]$, where $q$ is a vertex of the digon $B_i \cap B_j$ is an isosceles triangle whose base is of length $(1+\mu)\rho$, and $\rho=\rho_i=\rho_j$. Furthermore, since to decompose a core polygon into core triangles we can draw diagonals starting at any vertex of the polygon, we have that in case of equality in the inequality in Theorem~\ref{thm:main}, all sides and all diagonals of any core polygon are of equal length. From this we have that all core polygons are regular triangles, implying that all free digons in $\F$ are thick.

On the other hand, assume that all free digons in $\F$ are thick. Then, from Lemma \ref{3} it follows that any connected component of $\F$ contains congruent disks. Since an adjacent pair of disks defines a free digon, from this we have that, in a component consisting of disks of radius $\rho > 0$, the distance between the centers of two disks defining a shell triangle, and the edge-lengths of any core polygon, are equal to $(1+\mu)\rho$. Furthermore, since all disks centered at the vertices of a {core polygon are} touched by the same disk from inside, we also have that all core polygons in the component are regular $k$-gons of edge-length
$(1+\mu)\rho$, where $3 \leq k \leq 5$. This and the fact that any edge of a core polygon connects the vertices of an adjacent pair of disks yield that if the intersection of any two disks centered at two different vertices of a core polygon is more than one point, then it is a free digon. Thus, any diagonal of a core polygon in this component is of length $(1+\mu)\rho$, implying that any core polygon is a regular triangle, from which the equality in Theorem~\ref{thm:main} readily follows.

\section{Remarks and open questions}\label{sec:remarks}

\begin{rem}\label{rem:1}
If $\sqrt{3}-1 < \mu < 1$, then by Lemma~\ref{nocore}, $C(\F)=\emptyset$ for any $\mu$-arrangement $\F$ of order $\mu$.
\end{rem}

{
\begin{rem}\label{rem:2}
Observe that the proof of Theorem~\ref{thm:main} can be extended to some value $\mu > \sqrt{3}-1$ if and only if Lemma~\ref{shell} can be extended to this value $\mu$.
Nevertheless, from the continuity of the functions in the proof of Lemma~\ref{shell}, it follows that there is some $\mu_0 > \sqrt{3}-1$ such that the lemma holds for
any $\mu \in (\sqrt{3}-1,\mu_0]$. Nevertheless, we cannot extend the proof for all $\mu < 1$ due to numeric problems.
\end{rem}

Remark~\ref{rem:2} readily implies Remark~\ref{rem:3}.

\begin{rem}\label{rem:3}
There is some $\mu_0 > \sqrt{3}-1$ such that if $\mu \in (\sqrt{3}-1,\mu_0]$, and $\F$ is a $\mu$-arrangment of finitely many disks, then the total area of the disks is
\[
T \leq \frac{4\cdot \arccos({\frac{1+\mu}{2})}}{(1+\mu)\cdot \sqrt{(3+\mu)(1-\mu)}}\area(I(\mathcal{F}))+\area(O(\mathcal{F})),
\]
with equality if and only if every free digon in $\F$ is thick.
\end{rem}

\begin{conj}
The statement in Remark~\ref{rem:3} holds for any $\mu$-arrangement of finitely many disks with $\sqrt{3}-1 < \mu < 1$.
\end{conj}

Let $0 < \mu < 1$ and let $\F = \{ K_i : i=1,2,\ldots \}$ be a generalized Minkowski arrangement of order $\mu$ of homothets of an origin-symmetric convex body in $\Re^d$ with positive homogeneity. Then we define the (upper) density of $\F$ with respect to $U( \F)$ as
\[
\delta_U(\F) = \limsup_{R \to \infty} \frac{\sum_{B_i \subset R \BB^2} \area \left( B_i \right)}{\area \left( \bigcup_{B_i \subset R \BB^2} B_i \right)}.
\]
Clearly, we have $\delta(\F) \leq \delta_U(\F)$ for any arrangement $\F$.

Our next statement is an immediate consequence of Theorem~\ref{thm:main} and Remark~\ref{rem:3}.

\begin{cor}\label{cor:infinite}
There is some value $\sqrt{3}-1 < \mu_0 < 1$ such that for any $\mu$-arrangement $\F$ of Euclidean disks in $\Re^2$, we have
\[
\delta_U(\F) \leq \left\{ \begin{array}{l} \frac{2\pi }{\sqrt{3}(1+\mu)^2}, \hbox{ if } 0 \leq \mu \leq \sqrt{3}-1, \hbox{and}\\
\frac{4\cdot \arccos({\frac{1+\mu}{2})}}{(1+\mu)\cdot \sqrt{(3+\mu)(1-\mu)}}, \hbox{ if } \sqrt{3}-1 < \mu \leq \mu_0.
\end{array}
\right.
\]
\end{cor}

For any $0 \leq \mu < 1$, let $u, v \in \Re^2$ be two unit vectors whose angle is $\frac{\pi}{3}$, and let $\F_{hex}(\mu)$ denote the family of disks of radius $(1+\mu)$ whose set of centers is the lattice $\{ k u + m v: k,m \in \mathbb{Z} \}$. Then $\F_{hex}(\mu)$ is a $\mu$-arrangement, and by Corollary~\ref{cor:infinite}, for any
$\mu \in [0,\sqrt{3}-1]$, it has maximal density on the family of $\mu$-arrangements of positive homogeneity. Nevertheless, as Fejes T\'oth observed in \cite{FTL2} (see also \cite{BSz1} or Section~\ref{sec:intro}), the same does not hold if $\mu > \sqrt{3}-1$. Indeed, an elementary computation shows that in this case $\F_{hex}(\mu)$ does not cover the plane, and thus, by adding disks to it that lie in the uncovered part of the plane we can obtain a $\mu$-arrangement with greater density.

Fejes T\'oth suggested the following construction to obtain $\mu$-arrangements with large densities.
Let $\tau > 0$ be sufficiently small, and, with a little abuse of notation, let $\tau \F_{\hex}(\mu)$ denote the family of the homothetic copies of the disks in $\F_{\hex}(\mu)$ of homothety ratio $\tau$ and the origin as the center of homothety. Let $\F^1_{\hex}(\mu)$ denote the $\mu$-arrangement obtained by adding those elements of $\tau \F_{\hex}(\mu)$ to $\F_{\hex}(\mu)$ that do not overlap any element of it. Iteratively, if for some positive integer $k$, $\F^k_{\hex}(\mu)$ is defined, then let $\F^{k+1}_{\hex}(\mu)$ denote the union of $\F^k_{\hex}(\mu)$ and the subfamily of those elements of $\tau^{k+1} \F_{\hex}(\mu)$ that do not overlap any element of it. Then, as was observed also in \cite{FTL2}, choosing suitable values for $\tau$ and $k$, the value of $\delta_U(\F_{\hex}(\mu))$ can be approximated arbitrarily well by $\delta(\F^k_{\hex}(\mu))$. We note that the same idea immediately leads to the following observation.

\begin{rem}
The supremums of $\delta(\F)$ and $\delta_U(\F)$ coincide on the family of the $\mu$-arrangements $\F$ in $\Re^2$ of positive homogeneity.
\end{rem}

We finish the paper with the following conjecture.
%, and note that an affirmative answer to Conjecture~\ref{conj:2} readily yields an affirmative answer to Conjecture~\ref{conj:3}. Bocsánat: ugye pont most írtuk, hogy a kettő ekvivalens.

%\begin{conj}\label{conj:2}
%For any $\mu \in (\sqrt{3}-1, 1)$ and any $\mu$-arrangement $\mathcal{F}$ in $\Re^2$, we have $\delta_U(\mathcal{F}) \leq \delta_U(\F_{\hex}(\mu))$.
%\end{conj}

\begin{conj}\label{conj:3}
For any $\mu \in (\sqrt{3}-1, 1)$ and any $\mu$-arrangement $\mathcal{F}$ in $\Re^2$, we have $\delta(\mathcal{F}) \leq \delta_U(\F_{\hex}(\mu))$.
\end{conj}
}

{
\noindent
\textbf{Acknowledgments.}\\
The authors express their gratitude to K. Bezdek for directing their attention to this interesting problem, and to two anonymous referees for many helpful suggestions.
}

\end{document}